\documentclass{siamart190516}
\usepackage[T1]{fontenc}    
\usepackage{hyperref}       
\usepackage{url}            
\usepackage{booktabs}       
\usepackage{amsfonts}       
\usepackage{nicefrac}       
\usepackage{microtype}      

\usepackage{bm}
\usepackage{bbm}
\usepackage[utf8]{inputenc}
\usepackage[english]{babel}
\usepackage{amssymb}
\usepackage{amsmath}
\usepackage{amsfonts}
\usepackage{mathrsfs}
\usepackage{dsfont}
\usepackage{hyperref}
\usepackage{commath}
\usepackage{graphicx}
\usepackage{text comp}
\usepackage{mathtools}
\usepackage{afterpage}
\usepackage{appendix}

\usepackage{xcolor}

\newtheorem{defn}{Definition}[section]
\newtheorem{thm}{Theorem}[section]

\newcommand{\R}{\mathbb{R}}

\newcommand{\E}{\mathbb{E}}
\newcommand{\N}{\mathbb{N}}

\newcommand{\PP}{\mathbb{P}}

\newcommand{\1}{\mathbbm{1}}

\newcommand{\epsi}{\varepsilon}

\definecolor{brilliantrose}{rgb}{1.0, 0.33, 0.64}
\definecolor{amber}{rgb}{1.0, 0.75, 0.0}
\definecolor{amethyst}{rgb}{0.6, 0.4, 0.8}
\definecolor{carrotorange}{rgb}{0.93, 0.57, 0.13}
\definecolor{rosewood}{rgb}{0.4, 0, 0.04}
\definecolor{lincolngreen}{rgb}{0.11, 0.35, 0.02}
\definecolor{dukeblue}{rgb}{0, 0, 0.61}

\newcommand\new[1]{{\color{black}#1}}

\everymath{\displaystyle}

\author{
        Desmond John Higham%
            \thanks{%
           School of Mathematics,
           University of Edinburgh,
           Edinburgh, EH9 3FD, UK
           (\email{d.j.higham@ed.ac.uk}).
        }
        \and
        Henry-Louis de Kergorlay%
    \thanks{%
           School of Mathematics,
           University of Edinburgh,
           Edinburgh, EH9 3FD, UK
           (\email{hdekerg@ed.ac.uk})
           }
        }
        
\date{}
\title{Mean Field Analysis of Hypergraph Contagion Models\thanks{
\funding{Both authors were supported by Engineering and Physical
     Sciences Research Council grant EP/P020720/1.}}}

\begin{document}

\maketitle

\begin{abstract}
    We typically interact in groups, not just 
     in pairs. For this reason, it has recently been proposed that the spread of information, opinion or disease  
      should be modelled over a hypergraph rather than a standard graph.
       The use of hyperedges naturally allows for a 
       nonlinear rate of transmission, in terms of 
       \new{both the group size and the number of infected group members}, as is the case, for example, when social distancing is encouraged. 
       We consider a general class of individual-level,
        stochastic, 
        susceptible-infected-susceptible models 
          on a hypergraph, and focus on a mean field 
     approximation proposed in 
     [Arruda et al., Phys. Rev. Res., 2020].
     We derive spectral conditions under which the 
     mean field model predicts local or global stability of the infection-free state.
     We also compare these results with 
     \new{(a) a new condition that we derive for decay to zero in mean for the exact process}, 
      (b) 
      conditions for a different mean field approximation in 
      [Higham and de Kergorlay, Proc.\ Roy. Soc. A, 2021],
       and 
      (c) numerical simulations of the microscale model. 
\end{abstract}



\section{Motivation and Background}
\label{sec:mot}

Biological and social contagion processes 
can be used to model the way that opinions, rumours, ideas 
or diseases propagate through a community
\cite{DW05,KMS17}.
Traditionally a graph, or network, 
is used to represent the 
possible routes for 
person-to-person 
transmission
\cite{epidemicsSpread,HS2020,virusSpreadInNetworks,spreadingEigenvalue}.
Recent work has suggested that
it is beneficial to 
account directly for the  
higher-order group structures
that 
arise in human-to-human interactions, using hypergraphs
\cite{APM20,HdK21,heterogeneityHypergraph} 
or simplicial complexes \cite{simplicialSocialContagion,Torres_2020,DZLL21}.
Indeed, beyond-pairwise interactions are also relevant  
in many other social, economic and technological settings
\cite{ABAMPL21,BCILLPYP21,BASJK18,benson2016higher,ER06}.

In the context of opinion dynamics, an individual may 
be affected differently if multiple members of the same 
group (such as a workplace or household) express a view 
than if the same number of contacts from different groups
express that view \cite{simplicialSocialContagion}; this is an example of 
a \emph{majority effect} \cite{Maj18}.
Similarly, 
in the spread of a disease,
having multiple infected contacts in the same group
may lead to a different infection rate 
than having 
the same number of contacts across independent groups
\cite{heterogeneityHypergraph}.
For example, 
(unknowingly) 
sharing a photocopier 
with four 
infected colleagues may not be four times as risky as
sharing it with 
one infected colleague, if the item is cleaned regularly.
On the other hand, if there is a viral load threshold
\cite{Ebola15}
then sharing a car with 
four infected colleagues may be more than four times as risky
as sharing a car with one infected colleague.
Moreover the overall group size may have an effect---for a fixed classroom space, there may be a cutoff on the number students beyond which attempts at social distancing become ineffective.

For these reasons, it is natural to consider
a model of spreading that (a) uses information about the
groups present, rather than 
simply the resulting pairwise interactions, and 
(b) allows for the transmission rate to be 
a nonlinear function of the number of 
active individuals.
\new{Particular nonlinearities of interest are the 
concave, or 
\emph{collective suppression}, case \cite{HdK21}
and the threshold, or
\emph{collective contagion}, case
\cite{APM20,HdK21,simplicialSocialContagion,heterogeneityHypergraph}.
This leads to the
hypergraph-based model 
that we describe in section~\ref{sec:mod}, and 
the mean field approximation from \cite{APM20} that we describe in section~\ref{sec:mf}.
Sections~\ref{sec:exact} and \ref{sec:mf_analysis}
give stability analysis for 
the exact and mean field processes, respectively. 
In section~\ref{sec:compare} we compare
results with those for an alternative mean field model
of \cite{HdK21}.
An unusual feature of the mean field model in 
\cite{APM20} is that although it takes the form of 
a deterministic ODE system
with real-valued components, 
it evaluates 
the nonlinear infection rate 
function only at non-negative integer arguments, 
just as the exact stochastic model does.
This feature complicates the analysis, but we show that it
offers concrete benefits when the nonlinearity is concave.
Illustrative computational experiments are 
described in 
section~\ref{sec:compute}.
Corresponding results for a more general and flexible version of the 
hypergraph-based model are given in 
section~\ref{sec:multi}, and conclusions appear in section~\ref{sec:conc}.}

To be concrete, we describe the models
and analysis in the language of 
epidemiology, 
but we emphasize that the concepts and results 
are relevant in other scenarios.

\new{The main contributions of this work are:
\begin{itemize}
    \item for the exact model: a condition that guarantees decay to zero in mean of the infection level
(Theorem~\ref{thm: vanishing condition for exact mean field model}) 
and, for concave nonlinearity, a condition that guarantees
exponential decay to zero of the disease level
(Theorem~\ref{thm:exact_concave_decay}),
    \item for the mean field model of 
    \cite{APM20}:   
    a condition for local asymptotic stability of the 
    disease-free state (Theorem~\ref{thm: local stability}),
    and conditions for global asymptotic stability of the 
    disease-free state with collective suppression and collective infection nonlinearities 
    (Theorems~\ref{thm: global stability concave} and 
    \ref{thm: global stability collective}),
    \item extensions for the mean field model associated with a 
     more general multi-type model where the nonlinear infection rate 
     may depend on the category and size of the hyperedge
     (Theorems~\ref{thm:genralized local asymptotic stability},
     \ref{thm: generalized global stability concave} and 
     \ref{thm:generalized global stability collective}).
\end{itemize}
}

\section{Notation and Individual-level Model}
\label{sec:mod}

Before describing the model, we first introduce some definitions and notation.

A \emph{hypergraph} 
\cite{bretto2013hypergraph}
is a generalization 
of graph in which an edge, now called a \emph{hyperedge},
may join any number of vertices.
More formally,
a hypergraph is a
is a tuple $\mathcal H:=(V,E)$, where 
$V$ is a set of vertices and $E$ 
is a set of nonempty 
subsets of $V$ which specifies the hyperedges.

 We denote
  the number of nodes
 and hyperedges by 
 $n$  and $m$, respectively; that is, 
 $|V| =n$ and $|E| = m$.
 Assuming that the vertices and hyperedges have been ordered in some (arbitrary) way, 
 we use 
 ${\mathcal I}$
 to denote the 
 corresponding 
\emph{incidence matrix};
here 
${\mathcal I} \in \R^{n\times m}$ 
has 
${\mathcal I}_{i h}=1$ if node $i$ belongs to hyperedge $h$ and ${\mathcal I}_{i h}=0$ otherwise.


In our context the vertices represent individuals in a population of size $n$, and the hyperedges
record group interactions. For example,
a set of vertices may form a hyperedge if the 
corresponding individuals live in the same household, 
work in the same office or sing in the same choir.

Following the original idea in 
\cite{SISonHypergraphs},
which has also been 
studied in 
\cite{APM20,simplicialSocialContagion},
we use a continuous time Markov process
to track the propagation of disease through the population
\new{in a susceptible-infected-susceptible (SIS) framework}.
The state vector $X(t) \in \R^{n}$
is such that 
$X_i(t) = 1$ if vertex $i$ is infected at time 
$t$ and
$X_i(t) = 0$ otherwise.

We assume that the instantaneous recovery rate is given by a constant $\delta > 0$, and we let 
$\lambda_i(X(t))$ denote the state-dependent instantaneous infection rate for vertex $i$, given $X(t)$; that is,
\begin{align}
\PP \left( X_i(t+\epsi) = 1 \, | \, X(t) \right) 
    &=\begin{cases}
      \lambda_i(X(t)) \, \epsi + o(\epsi), &\text{~if~} X_i(t) = 0,\\
       1 - \delta \, \epsi + o(\epsi), &\text{~if~} X_i(t) = 1, 
    \end{cases}
     \label{eq:Xone}
\end{align}
and 
\begin{align}
\PP \left( X_i(t+\epsi ) = 0 \, | \, X(t) \right) 
    &=\begin{cases}
       \delta \, \epsi + o(\gamma), &\text{~if~} X_i(t) = 1,\\
       1 - \lambda_i(X(t)) \, \epsi  + o(\epsi), &\text{~if~} X_i(t) = 0.
    \end{cases}
    \label{eq:Xzero}
\end{align}
In this way, specifying the model reduces to 
defining the infection rates, $\lambda_i (X(t))$.
We mention that in the standard graph setting
\cite{epidemicsSpread,virusSpreadInNetworks,spreadingEigenvalue}, where interactions involve only pairs of vertices,
$\lambda_i (X(t))$ is taken to be proportional to the 
number of infected neighbours of vertex $i$ at time $t$.
Hence, in that case, the infection rate is linear in the number of infected neighbours.
As discussed in section~\ref{sec:mot}, we are interested in the setting of group interactions 
and possibly nonlinear 
infection rates.

Now,
writing $X_j$ rather than $X_j(t)$ for convenience,
we will assume that for a given vertex $i$, 
the contribution to the overall 
infection rate from a given hyperedge 
$h$ is 
\begin{equation}
\beta 
\,
{\mathcal I}_{i h} 
\,
f(
\sum_{j=1}^{n} 
{\mathcal I}_{j h}
X_j
).
\label{eq:hinf}
\end{equation}
\new{
Here, when ${\mathcal I}_{i h} =1$, so that $i$ is a member of the hyperdge, 
the argument 
passed to the function $f$ 
is the number of infected individuals
to which $i$ is exposed in this hyperedge.
Hence  
$f$ describes the
dependence of the infection rate
on the number of infected individuals.
The disease cannot spread unless there is at least one infected individual in the hyperedge, so we may assume throughout that $f(0) = 0$.
The factor 
$\beta$ 
in (\ref{eq:hinf}) 
represents the inherent 
infectiousness of the disease.}

For example, consider a one-hour meeting
between a predefined group of co-workers 
(forming a hyperedge) 
that takes place in a 
dedicated meeting
room.
Suppose further that, for this size of meeting room,
five
infected individuals, but no fewer, 
generate sufficient viral load to pass on the infection
(perhaps through 
airborne microdroplets
or through indirect contact).
Then a 
suitable nonlinearity in 
(\ref{eq:hinf})
could be 
$f(x) = c \, \max\{0,x -4\}$
or
$f(x) = c \, \1(x \ge 5)$, 
for some constant $c$; these are of 
collective contagion 
form 
\cite{APM20,HdK21,simplicialSocialContagion,heterogeneityHypergraph}.
Now suppose that the meeting room is in continual use,
for different groups (hyperedges) within the workforce,
and that vertex $i$ may
participate in several meetings.
We may then take 
the sum of (\ref{eq:hinf}) over all
groups (hyperedges).

For the purpose of analysis, it will be useful to 
categorize these hyperedges according to their size,
so we have categories
${\mathcal C}_2,\dots,{\mathcal C}_K$
with 
$h \in {\mathcal C}_k \Leftrightarrow |h| = k $.
\new{
Then the overall infection rate for vertex $i$ may be written 
\begin{equation}
\lambda_i(X(t)) 
= 
\beta 
\,
\sum_{k=2}^{K}
\sum_{h\in \mathcal C_k}
{\mathcal I}_{i h}
\,
f (
\sum_{j=1}^{n} 
{\mathcal I}_{j h}
X_j
).
\label{eq:Xinf}
\end{equation}
}

It is natural to generalize the expression 
(\ref{eq:Xinf}) to incorporate
different types of hyperedge;
for example these may correspond
to groups that congregate in various sizes of 
classroom, workspace,
residence, or
vehicle,
and groups that interact 
through various kinds of 
sports or leisure activities.
Each different type of hyperedge 
may be given its own  
function 
$f$ to quantify the
dependence of the infection rate
on the number of infected individuals in that setting, and
$\lambda_i(X(t)) $
in (\ref{eq:Xinf}) would generalize to include the sum over all contributions.
The analysis below extends readily to this case, at the expense of 
notational complexity.
For the sake of clarity, we therefore state and prove results for the 
one-type model (\ref{eq:Xinf}), and in Section~\ref{sec:multi}
we explain how the results extend to the multi-type model.

In \cite{APM20} the authors considered a model of the form  
(\ref{eq:Xone}), (\ref{eq:Xzero}), (\ref{eq:Xinf})  
with a particular collective contagion nonlinearity $f$.
(More precisely, the model in \cite{APM20} is covered by the multi-type
setting of Section~\ref{sec:multi}.)
A first order, or mean field, approximation to the individual-level 
model 
was derived in 
\cite{APM20}, and the dynamical behaviour of the 
resulting ODE system was investigated numerically.
In the next section we 
describe this mean field approach for a general 
nonlinearity, $f$. 
\new{Later, in section~\ref{sec:compare}, we compare the performance of this model with another, simpler, mean field approximation 
that was derived and studied in \cite{HdK21} based on the
idea of commuting the order of $\E$ and $f$.}

\section{Mean Field Hypergraph Models}
\label{sec:mf}

The rate of infection expressed in (\ref{eq:Xinf}) is random.
\new{To make large-scale simulations tractable, and to 
facilitate analysis, it is natural to focus on  
the evolution of the the expected processes $(p_i(t))_{t\geq 0}:=(\E[X_i(t)])_{t\geq 0}$, $i\in\{1,2,\dots,n\}$.
Substituting the random rates of infection by their expectation,
gives}
\begin{equation}\label{eq: exact mf}
\frac{d p_i}{dt}=\E[
\lambda_i(X(t))
](1-p_i)-\delta p_i.
\end{equation}
Taking expected values in (\ref{eq:Xinf}), the expected rate of infection may be written
\begin{eqnarray}
\E[
\lambda_i(X(t))
]
&=&
\beta \sum_{k=2}^K\sum_{h\in \mathcal C_k}\mathcal I_{i h}\E[f(\sum_{j=1}^n X_j \mathcal I_{j h})] \nonumber \\ 
&=& 
\beta \sum_{k=2}^K\sum_{h\in \mathcal C_k}\mathcal I_{i h}\sum_{l=1}^k f(l)\mathbb P(\sum_{j=1}^n X_j \mathcal I_{jh} = l).
\label{eq:mean_exact}
\end{eqnarray}
\new{This expression defines the expected rate exactly, but it does not 
appear to be amenable to numerical simulation.
In \cite{APM20}, an approximation was introduced by  
 assuming independence of the $X_j$, giving}
\begin{equation}\label{eq:indep}
\mathbb P(\sum_{j=1}^n X_j \mathcal I_{jh} = l)\approx \Psi(h,l) := \sum_{J_l\subset h}\prod_{j\in J_l}p_j\prod_{j\in h\setminus J_l}(1-p_j),
\end{equation}
where $J_l$ runs over all possible subsets of nodes of hyperedge $h$, of size $l$.
To avoid cumbersome notation, we do not explicitly denote the dependence of $p_j$ on $t$ or the dependence 
of $\Psi(h,l)$ on the $p_j$.
With this approximation, the expected processes $P(t):=(p_i(t))_{i=1}^n$ satisfy the deterministic ODE system 
\begin{equation}\label{eq:dynamical system}
\frac{d P(t)}{dt}=g(P(t)),
\end{equation}
where $g:\R^n\to\R^n$ is defined by
\begin{equation}
g_i(P(t)):=\beta\sum_{k=2}^{K}\sum_{h\in \mathcal C_k}\mathcal I_{i h}(\sum_{l=1}^kf(l)\Psi(h,l))(1-p_i(t))-\delta p_i(t).
\label{eq:gidef}
\end{equation}
We emphasize that, with a slight abuse of notation, $p_i(t)$ is now being used to 
denote a mean field approximation to 
$\E[X_i(t)]$. We also note that the factors 
$\mathcal I_{i h}$ in (\ref{eq:gidef}) implicitly depend on
$k$ through the hyperedge 
constraint $h\in \mathcal{C}_k$.
To make the model physically reasonable we 
assume that the initial conditions satisfy 
$0 \le p_i(0) \le 1$ for $ i = 1,\ldots,n$,
and we note that 
$0 \le p_i(t) \le 1$ for $ i = 1,\ldots,n$ then follows for all $t > 0$.

\new{This mean field ODE was derived and studied numerically
in \cite{APM20} with 
an emphasis on first-and second-order transitions,
bistability and 
hysteresis.
Our aim in this work is to 
derive analytical results that address a more fundamental question: under what conditions will the disease will die out?
We do this by studying the local and global stability of the 
disease-free state.
In the next section, we show that it is possible to analyse the 
exact expected process, and in section~\ref{sec:mf_analysis}
we move on to the 
mean field approximation
(\ref{eq:dynamical system})--(\ref{eq:gidef}).
}

\section{The Exact Expected Process}\label{sec:exact}
Here we show that, while the exact equation describing the dynamics of the expected processes (\ref{eq: exact mf}) does not seem to be amenable to numerical simulation, an upper bound argument allows us to derive vanishing conditions. In the following analysis, and throughout the remaining sections, we define the symmetric matrix 
$W\in \R^{n\times n}$ by 
\[
W_{i j}:=\sum_{k=2}^K\sum_{h\in \mathcal C_k}\mathcal I_{ih}\mathcal I_{jh},
\]
so that $W_{ij}$
records the number of hyperedges containing both nodes $i$ and $j$. Given a symmetric matrix $A$, we let $\lambda(A)$ denote its largest eigenvalue.
\new{Recall from (\ref{eq: exact mf}) and 
(\ref{eq:mean_exact})} that we have, for $i\in\{1,2,\dots,n\}$,
\begin{equation}\label{eq: exact mean field model}
\frac{d p_i}{dt}=\beta (1-p_i)\sum_{k=2}^K\sum_{h\in \mathcal C_k}\mathcal I_{i h}\E[f( X_j \mathcal{I}_{j h})]-\delta p_i.
\end{equation}
We also define the constant $c_f$ as follows.

\begin{defn}
Let
\[
 c_f := 
    \max_{x\in \{1,2,\dots,K\}}
       \frac{ f(x) } { x}.
    \]
\end{defn}
With this definition, we have 
\begin{align*}
\beta \sum_{k=2}^K\sum_{h\in \mathcal C_k}\mathcal I_{ih}\E[f(\sum_{j=1}^nX_j\mathcal I_{jh})] &\leq \beta \, c_f\sum_{k=2}^K\sum_{h\in \mathcal C_k}\mathcal I_{ih}\sum_{j=1}^np_j\mathcal I_{jh}\\
    &= \beta \, c_f\sum_{j=1}^nW_{ij}p_j.
\end{align*} 
Hence, from 
(\ref{eq: exact mean field model}), 
we have the following differential inequalities, for  $i\in \{1,2,\dots,n\}$, 
\begin{equation}\label{eq: ODI}
\frac{d p_i}{dt}\leq  c_f \, \beta \, \sum_{j=1}^nW_{ij}p_j(1-p_i)-\delta p_i.
\end{equation}
This system of differential inequalities can be analyzed after invoking the result below.
\begin{thm}[\cite{ODI}]\label{thm: ODI}
Suppose that $u$ satisfies $u'(t)\leq f(u(t),t)$ and $y$ satisfies $y'(t)=f(y(t),t)$, with boundary condition $u(t_0)=y(t_0)$. Then \[\begin{cases}
\forall\ t<t_0,\ u(t)\geq y(t)\\
\forall\ t>t_0,\ u(t)\leq y(t).
\end{cases}\]
\end{thm}
This result readily extends to a system of differential inequalities as in $(\ref{eq: ODI})$, as follows.
\begin{thm}\label{thm: ODI system result}
Suppose that $\{u_i\}_{i=1}^n$ satisfies
\[
\forall\ i\in \{1,2,\dots,n\},\ \forall\ t\in \R,\ u_i'(t)\leq f_i(u(t),t),
\]
that $\{y_i\}_{i=1}^n$ satisfies
\[
\forall\ i\in \{1,2,\dots,n\},\ \forall\ t\in \R,\ y_i'(t)= f_i(y(t),t),
\]
and that for all $i\in \{1,2,\dots,n\}$, $u_i(t_0)=y_i(t_0)$.
Then for all $i\in \{1,2,\dots,n\}$
\[\begin{cases}
\forall\ t<t_0,\ u_i(t)\geq y_i(t)\\
\forall\ t>t_0,\ u_i(t)\leq y_i(t).
\end{cases}\]
\end{thm}
For our purposes, we have the following. 
\begin{corollary}\label{cor: ODI system result}
If for all $i\in \{1,2,\dots,n\}$, $u_i'(t)\leq g_i(u(t))$, $y_i'(t)=g_i(y(t))$ and $u_i(0)=y_i(0)$, then for all $i\in \{1,2,\dots,n\}$ and all $t\geq 0,$ $u_i(t)\leq y_i(t).$
\end{corollary}

 We are interested in finding conditions under which the spread of the disease predicted by $(\ref{eq: exact mean field model})$ vanishes as $t\to\infty$. By Corollary~\ref{cor: ODI system result} and $(\ref{eq: ODI})$, it suffices to find such conditions for the following, more simple, model
\begin{equation}\label{eq:concave upper bound dynamical system}
    \frac{d P(t)}{dt}=\widetilde{g}(P(t)),
\end{equation}
where $\widetilde{g}:\R^n\to\R^n$ is defined by
\begin{equation}\label{eq:tildegidef}
    \widetilde{g}_i(P(t)):=\beta \, c_f \, \sum_{j=1}^nW_{ij}p_j(t)(1-p_i(t))-\delta p_i(t).
\end{equation}
This system can be analysed \new{by appealing to \cite[Theorem $6.4$]{HdK21} in the case where the infection function is the identity (which, in particular, is concave)}, and we deduce the following result. 
\begin{thm}[Extinction in mean for the exact process]\label{thm: vanishing condition for exact mean field model}
If 
\begin{equation}
\frac{
\beta \,  c_f \, \lambda(W) 
}
{
\delta
}<1,
\label{eq:spec_exact_mean}
\end{equation}
then $0$ is a globally asymptotically stable equilibrium for $(\ref{eq:concave upper bound dynamical system})$ and hence for $(\ref{eq: exact mean field model})$, 
that is, for all $i\in \{0,1,\dots,n\}$ and all initial conditions, $\lim_{t\to \infty}p_i(t)=0$ in $(\ref{eq: exact mean field model})$. 
\end{thm}

Considering the collective suppression case, where $f$ is concave and $f(0)=0$, we have for all $x\in \N$, $f(x)\leq f(1)x$; hence $c_f =f(1)$. 

For a collective contagion model
of the form $f(x):=c_2\, \1(x\geq c_1)$ for some constants 
\new{$c_1 \geq 1$ and $c_2>0$}, we have for all $x\in \N$, $f(x)\leq c_2 \, x/ c_1$, hence $c_f \le c_2/c_1$.

Theorem~\ref{thm: vanishing condition for exact mean field model} gives a practical condition for the exact model that guarantees 
decay to zero in mean of the infection level of every component.
In the next section we seek similar results for the mean field 
approximation (\ref{eq:dynamical system})--(\ref{eq:gidef}).
This allows us (a) to judge the accuracy of this 
mean field approximation in terms of 
a corresponding spectral threshold, and 
(b) to get insights into the behaviour of a system that 
can be
simulated directly. Also,
in section~\ref{sec:compare} 
we use this analysis to 
compare predictions against those of
the alternative mean field model from \cite{HdK21}.

\section{Analysis of Mean Field Hypergraph Model}
\label{sec:mf_analysis}


Here we analyze the mean field model described in (\ref{eq:dynamical system})--(\ref{eq:gidef}). We find conditions for local and global asymptotic stability of the disease-free state of the process, considering various assumptions on the infection function $f$, including collective contagion and  collective suppression cases. 
Our first result is a spectral condition for local asymptotic stability.

\begin{thm}[General condition for local asymptotic stability]\label{thm: local stability}
If 
\begin{equation} \label{eq:mflocal}
    \frac{\beta \, f(1) \, \lambda(W) }  {\delta} <1,
\end{equation}
then $0\in \R^n$ is a locally asymptotically stable equilibrium for (\ref{eq:dynamical system})--(\ref{eq:gidef}).
\end{thm}
\begin{proof}
We have  
$g(0)=0$
in 
(\ref{eq:dynamical system})--(\ref{eq:gidef}), so $0 \in \R^n$ is an equilibrium. 
From a standard linearization 
result~\cite{Ve90}, 
local asymptotic stability 
follows if every eigenvalue of the Jacobian matrix $\nabla g(0)$ has a negative real part.
 For $j_0 \neq i$ we compute
\begin{equation}
    \frac{\partial g_i}{\partial p_{j_0}}=
    \beta\sum_k\sum_{h\in \mathcal C_k}\mathcal I_{i h}\mathcal 
    I_{j_0 h}\sum_{l=1}^kf(l)\frac{\partial \Psi}{\partial p_{j_0}}(h,l)(1-p_i), \\
    \label{eq:jac1}
 \end{equation}  
   and along the diagonal
   \begin{equation}
     \frac{\partial g_i}{\partial p_{i}}
     =
    \beta\sum_k\sum_{h\in \mathcal C_k}\mathcal I_{i h}\sum_{l=1}^kf(l)\frac{\partial \Psi}{\partial p_{i}}(h,l)(1-p_i)-\beta\sum_k\sum_{h\in \mathcal C_k}\mathcal I_{i h}\sum_{l=1}^kf(l)\Psi(h,l)-\delta. 
 \label{eq:jac2}
\end{equation}
We see that
$\nabla g(0) = B-\delta I$, where 
\[B_{ij} :=\beta\sum_k\sum_{h\in \mathcal C_k}\mathcal I_{i h}\mathcal I_{j h}\sum_{l=1}^kf(l)\frac{\partial \Psi}{\partial p_{j}}(h,l)|_{P=0}.
\]

Let $h\in E$ be a hyperedge, let $j_0\in V=\{1,2,\dots,n\}$ be a node of the hypergraph, and let $\tilde h:=h\setminus\{j_0\}.$ Using (\ref{eq:indep}), we find
\[
    \frac{\partial \Psi(h,l)}{\partial p_{j_0}}=\mathcal I_{j_0 h}\left(\sum_{J_{l-1}\cup\{j_0\}\subset h}\prod_{j\in J_{l-1}}p_j\prod_{j\in \tilde h\setminus J_{l-1}}(1-p_j)-\sum_{J_l\subset \tilde h}\prod_{j\in J_l}p_j\prod_{j\in \tilde h\setminus J_l}(1-p_j)\right)
    \]
    and hence
\[
 \frac{\partial \Psi(h,l)}{\partial p_{j_0}}|_{P=0}
    = 
\mathcal I_{j_0h}\delta_{l 1},
\]
with Kronecker delta notation, so that $\delta_{xy}=1$ if $x=y$, and $\delta_{xy}=0$ otherwise.
It follows that $B_{ij}=\beta f(1)\sum_{k=2}^{K-1}\sum_{h\in \mathcal C_k}\mathcal I_{ih}\mathcal I_{jh}=\beta f(1) W_{ij}$.

We see that $\nabla g(0)$ is symmetric and
$ \lambda(\nabla g(0))=  
 \lambda(B - \delta I) = \lambda(W) f(1) \beta -\delta$. So it suffices for local 
asymptotic stability that $\lambda(W) f(1)\beta/\delta<1$, as stated.
\end{proof}

The above result has the advantage that it 
does not require specific assumptions on the infection model. However, it is relevant only when
the initial proportion of infected individuals is sufficiently small. We now consider particular infection models, with the aim of
constructing a global asymptotic stability result.
\new{As discussed in section~\ref{sec:mod}, and in more detail 
in \cite{HdK21},} 
two cases of practical relevance are: a \textit{collective suppression model}, characterized by a concave infection function $f$, and a \textit{collective contagion model}, 
characterized by $f(i) = 0$ for $i = 0,1,\ldots,m-1$ with some $m \ge 2$. In the latter case, 
the disease may only start spreading in a hyperedge if the number of infected individuals
in that hyperedge reaches a critical threshold value, $m$.
When the infection function is concave, the local asymptotic stability result obtained in Theorem~\ref{thm: local stability} extends to the case of global asymptotic stability.

\begin{thm}[Global asymptotic stability for a collective suppression model]\label{thm: global stability concave}
Suppose that $f$ is concave. 
\new{If 
the spectral bound (\ref{eq:mflocal})
holds,} 
 then $0\in \R^n$ is globally asymptotically stable for (\ref{eq:dynamical system})--(\ref{eq:gidef}).
\end{thm}

To prove this theorem, we first introduce a few preliminary results. Let $h\in E$ be a hyperedge. 
\new{To avoid cumbersome notation we assume that $|h|=K$ and let the nodes in $h$ be $\{1,2,\dots,K\}$. 
Any other hyperedge could be analyzed in similar way,
but, for example, in Lemma~\ref{lemma: rewrite sum of partials in terms of x_k} below we would then need to write 
$\{ 1, 2, \ldots, K\} \setminus \{i\}$ rather than 
$\{ 1, 2, \ldots, K-1\}$.
Also, to streamline the presentation, we use the 
additional notation $[K] = \{ 1, 2, \ldots, K\}$ where convenient.
}

We seek to estimate the spectrum of the Jacobian matrix of $g$ at all points $P\neq 0$, as in \cite[Theorem 6.4]{HdK21}.
Hence, from
(\ref{eq:jac1})--(\ref{eq:jac2}), 
we need to estimate $\sum_{l=1}^kf(l)\frac{\partial \Psi}{\partial p_K}(h,l)$. To this end, let us first rewrite $\Psi(h,l)$ according to the following lemma.

\begin{lemma}\label{lemma: bernouilli sum}
Let $\{z_i\}_{i=1}^K$ be a set of independent Bernoulli random variables  such that for each $i\in \{1,2,\dots,K\}$
$$
z_i=\begin{cases}1&\text{ with probability }p_i\\
0&\text{ with probability }1-p_i.
\end{cases}
$$
For every $l\in \{1,2,\dots,K\}$
$$
\mathbb P(\sum_{i=1}^Kz_i=l)=\sum_{J_l\subset [K]}\prod_{j\in J_l}p_j+\sum_{k=l+1}^K(-1)^{k-l}{l+k\choose l}\sum_{J_k\subset [K]}\prod_{j\in J_k}p_j,
$$
where $J_l$ runs over all possible subsets of $[K]:=\{1,2,\dots,K\}$ of size $l$.
\end{lemma}
\begin{proof}
Letting $A_i:=\{z_i=1\}$ for every $i\in\{1,2,\dots,K\},$ we have for $l\in \{1,2,\dots,K\}$
\begin{equation}\label{eqn: sum of bernouilli}
    \mathbb P(\sum_{i=1}^Kz_i=l)=\sum_{J_l\subset [K]}\mathbb P\left((\cap_{i\in J_l}A_i)\cap (\cap_{j\in [K]\setminus J_l} A_{j}^c)\right),
\end{equation}
where $A^c$ denotes the complement of event $A$.

Let us estimate $\mathbb P(A\cap A_{l+1}^c\cap\dots\cap A_K^c)$, where $A:=A_1\cap\dots\cap A_l$, thus considering without loss of generality the case $J_l=\{1,2,\dots,l\}$. Define the induced probability measure from $A$ by 
$$
\mathbb P_A(B):=\frac{\mathbb P(A\cap B )}{\mathbb P(A)}.
$$
By the inclusion-exclusion principle, we have
\begin{align*}
    &\mathbb P(A\cap A_{l+1}^c\cap\dots\cap A_K^c)\\
    &=\mathbb P(A)\mathbb P_A(A_{l+1}^c\cap\dots\cap A_K^c)\\
    &= \mathbb P(A)\left(1-\sum_{i=l+1}^K\mathbb P_A(A_i)+
    \right. \\
    &\mbox{}~~~~~~
    \left.
    \sum_{l+1\leq i< j\leq K}\mathbb P_A(A_i\cap A_j)+\dots+(-1)^{K-l}\mathbb P_A(A_{l+1}\cap\dots\cap A_K)\right)\\
    &=\mathbb P(A)-\sum_{i=l+1}^K\mathbb P(A\cap A_i)\\
    &\mbox{}~~~
    +\sum_{l+1\leq i< j\leq K}\mathbb P(A\cap A_i\cap A_j)+\dots+(-1)^{K-l}\mathbb P(A\cap A_{l+1}\cap\dots\cap A_K).
\end{align*}
Spanning over all $J_l$, we see that for each $i\in \{l+1,\dots,K\},$ $\mathbb P(A\cap A_i)$ is substracted in (\ref{eqn: sum of bernouilli}) exactly $l+1$ times. Indeed it is counted once  for each $J_l$ satisfying
\begin{align*}
&\exists\ i_0\in[K]\setminus J_l,\ (\cap_{j\in J_l}A_j)\cap A_{i_0}=A\cap A_i=A_1\cap \dots \cap A_l\cap A_i\\
\Leftrightarrow\ &\exists\ i_0\in[K]\setminus J_l,\ J_l\cup\{i_0\}=\{1,2,\dots,l,i\},
\end{align*} 
which yields ${l+1\choose l}=l+1$ possible choices for $J_l$. Likewise, for every $1\leq i<j\leq K$, $\mathbb P(A\cap A_i\cap A_j)$ is added exactly ${l+2\choose l}$ times in (\ref{eqn: sum of bernouilli}), and more generally every $(-1)^{k-l}\mathbb P(A\cap A_{j_{l+1}}\cap \dots A_{j_{k}})$ is added in (\ref{eqn: sum of bernouilli}) exactly ${l+k\choose l}$ times. This, together with the independence of the $z_i$, yields the claimed formula.
\end{proof}

Using Lemma~\ref{lemma: bernouilli sum}, we deduce the following lemma.
\begin{lemma}\label{lemma: rewrite sum of partials in terms of x_k}
We have
$$
\sum_{l=1}^Kf(l)\frac{\partial \Psi}{\partial p_K}(h,l)= f(1) +\sum_{k=2}^{K-1}x_k (\sum_{J_{k-1}\subset [K-1]}\prod_{j\in J_{k-1}}p_j),
$$
where, for $k\geq 2$, we let
\begin{equation}
x_k:=f(k)+\sum_{l=1}^{k-1}(-1)^{l}{2k-l\choose k-l}f(k-l).
\label{eq:xkdef}
\end{equation}
\end{lemma}
\begin{proof}
By Lemma~\ref{lemma: bernouilli sum}, we can rewrite $\Psi(h,l)$ as
\begin{equation}\label{estimated prob}
\Psi(h,l) = \sum_{J_l\subset [K]}\prod_{j\in J_l}p_j+\sum_{k=l+1}^K(-1)^{k-l}{l+k\choose l}\sum_{J_k\subset [K]}\prod_{j\in J_k}p_j.
\end{equation}

From $(\ref{estimated prob})$, we find for $K\geq 2$ (for $K=1$, the partial derivative is equal to $1$)
that $\partial \Psi(h,l)/\partial p_K $
takes the form 
\[
\begin{cases} 
1 +\sum_{k=1}^{K-1}(-1)^{k}{l+k+1\choose l}\sum_{J_k\subset [K-1]}\prod_{j\in J_k}p_j,&\text{ if }l=1\\
\sum_{J_{l-1}\subset [K-1]}\prod_{j\in J_{l-1}}p_j+\sum_{k=l}^{K-1}(-1)^{k-l+1}{l+k+1\choose l}\sum_{J_k\subset [K-1]}\prod_{j\in J_k}p_j,&\text{ otherwise.}
\end{cases}
\]
Multiplying the above by $f(l)$, summing over $l$ and grouping the terms according to each $\sum_{J_k\subset [K-1]}\prod_{j\in J_k}p_j$, we find
that
\[
\sum_{l=1}^Kf(l)\frac{\partial \Psi}{\partial p_K}(h,l)
\]
may be written 
\[
 f(1)\ +\ \sum_{k=2}^{K-1}(f(k)+\sum_{l=1}^{k-1}(-1)^{l}{2k-l\choose k-l}f(k-l))(\sum_{J_{k-1}\subset [K-1]}\prod_{j\in J_{k-1}}p_j).
\]
\end{proof}

\begin{lemma}\label{lemma: x_k negative}
Suppose that $f$ is concave, then $x_k\leq 0$ 
\new{in (\ref{eq:xkdef})} for all $k\geq 2.$
\end{lemma}
\begin{proof}
It is clear that $x_2\leq 0$, so it remains to show that $x_k\leq 0$ for all $k\geq 3$.
Letting $C_l:={2k-l\choose k-l}$ for fixed $k\geq 3$, we have
\begin{align*}
&\sum_{l=1}^{k-1}(-1)^lC_lf(k-l) =\\
&\begin{cases}
-\left((C_1f(k-1)-C_2f(k-2))+\dots+C_{k-1}f(1)\right), 
  &k\equiv 0 \ \mathrm{mod}\ 2\\
-\left((C_1f(k-1)-C_2f(k-2))+\dots+(C_{k-2}f(2)-C_{k-1}f(1))\right), &k\equiv 1\ \mathrm{ mod }\ 2.
\end{cases}
\end{align*}
Since $C_lf(k-l)$ is decreasing in $l$, it suffices to show, for all $k\geq 3$, that
\begin{align*}
    f(k)&\leq\begin{cases}
    C_{k-1}f(1)=(k+1)f(1),
    &k\equiv 0 \ \mathrm{mod}\ 2\\
    C_{k-2}f(2)-C_{k-1}f(1)={k+2\choose 2}f(2)-(k+1)f(1),
     &k\equiv 1 \ \mathrm{mod}\ 2.
    \end{cases}
\end{align*}

By the concavity of $f$ and $f(0)=0$, we see that the slopes $f(k)/k$ are decreasing in $k\geq 1$. Hence we already have that $f(k)\leq (k+1)f(1)$, and it remains to show that for all $k\geq 3$
$$
f(k)\leq {k+2\choose 2}f(2)-(k+1)f(1).
$$
Dividing both sides of the above inequality by $k$, we see that the LHS decreases in $k$, while the RHS increases in $k$; hence it suffices to show the inequality for $k=3$.
By the concavity of $f$,
$
f(3)\leq 2f(2)-f(1),
$
hence ${5\choose 2}f(2)-4f(1)\geq 4f(3)$.
\end{proof}
\begin{proof}[Proof of Theorem~\ref{thm: global stability concave}]
From the global asymptotic stability result in \cite[Lemma $1'$ ]{GSAcriterion} 
it is sufficient to show that 
all eigenvalues of the symmetric matrix 
\[
 (\nabla g(P))^{(S)}:= ( \nabla g(P) + \nabla g(P)^T)/2
\]
are strictly less than $0$, for all $P\neq 0$.

From (\ref{eq:jac1})--(\ref{eq:jac2}), 
using Lemma $6.3$ in \cite{HdK21} with diagonal matrix given by $\Lambda_{ii}:=\beta\sum_k\sum_{h\in \mathcal C_k}\mathcal I_{i h}\sum_{l=1}^kf(l)\Psi(h,l)\geq 0$, we deduce that
$$
\lambda(\nabla g(P)^{(S)})\leq \lambda(B^{(S)}-\delta I),
$$
where $B_{ij}:=\beta\sum_k\sum_{h\in \mathcal C_k}\mathcal I_{i h}\mathcal I_{j h}\sum_{l=1}^kf(l)\frac{\partial \Psi}{\partial p_{j}}(h,l)(1-p_i)$.

Since $f$ is concave, we know by Lemmas~\ref{lemma: rewrite sum of partials in terms of x_k} and \ref{lemma: x_k negative}, that for all $j\in \{1,2,\dots,n\}$, $\sum_{l=1}^Kf(l)\frac{\partial \Psi}{\partial p_j}(h,l)\leq f(1)$, 
from which it follows that
$$
0\leq B_{ij}^{(S)}\leq \beta f(1) \sum_{k=2}^{K-1}\sum_{h\in \mathcal C_k}\mathcal I_{ih}\mathcal I_{jh}=\beta f(1) W_{ij}.
$$
Hence $\lambda((\nabla g(P))^{(S)})\leq \lambda(\beta f(1)W-\delta I)$, and it suffices that $\lambda(W)\frac{f(1)\beta}{\delta}<1$, 
which completes the proof of Theorem~\ref{thm: global stability concave}.
\end{proof}
Applying Lemmas~\ref{lemma: rewrite sum of partials in terms of x_k} and \ref{lemma: x_k negative} to the identity function (which is concave), we find
$$
\sum_{l=1}^Kl\frac{\partial \Psi(h,l)}{\partial p_K}\leq 1.
$$
Hence for all choices of $f$, if $c_f>0$ is such that for all $x\in \N$, $f(x)\leq c_f x$, then 
$$
\sum_{l=1}^{K}f(l)\frac{\partial \Psi(h,l)}{\partial p_K}\leq c_f \sum_{l=1}^Kl\frac{\partial \Psi(h,l)}{\partial p_K}\leq c_f.
$$
In particular for a collective infection model, where $f(x):=c_2\1(x\geq c_1)$, we deduce that
$$
\sum_{l=1}^{K}f(l)\frac{\partial \Psi(h,l)}{\partial p_K}\leq\frac{c_2}{c_1},
$$
from which the next theorem follows.

\begin{thm}[Global asymptotic stability for a collective contagion model]\label{thm: global stability collective}
Suppose that $f(x):=c_2\1(x\geq c_1).$ If 
\begin{equation}
\frac{ \beta \, c_2 \, \lambda(W)} { \delta \, c_1 } < 1,
\label{eq:gscoll}
\end{equation}
then $0\in \R^n$ is globally asymptotically stable for 
(\ref{eq:dynamical system})--(\ref{eq:gidef}).
\end{thm}

The proof of Theorem~\ref{thm: global stability concave} above may be used to establish this result, substituting $f(1)$ by $c_2/c_1$ everywhere.


\section{Comparison with Alternative Mean Field Model and Exact Model}
\label{sec:compare}
As mentioned in section~\ref{sec:mod}, 
an alternative mean field approximation model was 
introduced and studied in \cite{HdK21}. This is given by
\begin{equation}\label{eq:alternative dynamical system}
\frac{d P(t)}{dt}=\widehat{g}(P(t)),
\end{equation}
where $\widehat{g}_i:\R^n\to \R$ is defined by
\begin{equation}
\widehat{g}_i(P(t)):=\beta\sum_{h\in E}\mathcal I_{i h}f\left(\sum_{j=1}^{n}p_j(t)\mathcal I_{j h}\right)(1-p_i(t))-\delta p_i(t).
\label{eq:alternative gidef}
\end{equation}
\new{The key approximation in the derivation of this model is to take
the expectation operation 
inside the function $f$. Comparing 
(\ref{eq:dynamical system})--(\ref{eq:gidef}) and
(\ref{eq:alternative dynamical system})--(\ref{eq:alternative gidef}),
}
 one major difference is that while the infection function $f$ is only \new{evaluated} over integers in $g$, it is \new{evaluated} on a continuous domain in $\widehat{g}$. 
 \new{This leads to different factors in the 
 spectral bounds.}
 Indeed, suppose that $f$ is concave.
 \new{Theorem~\ref{thm: global stability concave} tells us that 
 the solution of mean field approximation model given by $g$ 
 in (\ref{eq:dynamical system})--(\ref{eq:gidef}).
  vanishes if $ \beta \, f(1) \, \lambda(W) /\delta<1$.
  For the model defined by $\widehat{g}$ in
  (\ref{eq:alternative dynamical system})--(\ref{eq:alternative gidef}), \cite[Theorem~6.4]{HdK21} 
  gives the condition 
  \begin{equation}\label{eq:mflocal2}
  \frac{\beta \, f'(0) \, \lambda(W) } { \delta } <1
  \end{equation}
  \new{for global asymptotic stability.}
 In this concave setting, the slopes $x\mapsto (f(x)-f(0))/(x-0)$ are decreasing in $x>0$. Since $f(0)=0$, we deduce that $f'(0)=\lim_{x\to 0} f(x)/x \geq f(1)$ always holds true. 
 Hence, for the mean field model 
  (\ref{eq:dynamical system})--(\ref{eq:gidef})
  we have 
 \emph{a less restrictive sufficient condition} for vanishing of the disease. 
 Moreover, the following theorem shows that a similar condition 
 controls the behavior of the exact solution, and hence, in this sense,
   (\ref{eq:dynamical system})--(\ref{eq:gidef})
   gives a more accurate approximation
   than 
   (\ref{eq:alternative dynamical system})--(\ref{eq:alternative gidef})
   in the concave case.}
   
   \begin{thm}\label{thm:exact_concave_decay}
Suppose that $f$ is concave in the \new{mean field} model given by (\ref{eq:dynamical system})--(\ref{eq:gidef}). 
\new{
Also assume for simplicity that each node has the same,
independent, initial infection probability denoted 
by $i_0$; that is, for $j = 1,2,\ldots,n$,
\begin{equation} 
\PP( X_j(0)  = 1) = i_0.
\label{eq:izero}
\end{equation}
}
Then
$$
\mathbb P\left(\sum_{i=1}^nX_i(t)>0\right)\leq n \,
i_0 \, \exp \left( (\beta \, f(1) \, \lambda(W)-\delta)t \right).
$$
Hence, if $ \beta \,  f(1) \, \lambda(W) / \delta <1$ the disease vanishes at an exponential rate.
\end{thm}
\begin{proof}
 This result may be proved using the arguments in the proof of   \cite[Theorem $8.1$]{HdK21}, noticing that we can substitute $f'(0)$ by $f(1)$. 
\end{proof}


The following corollary also holds, analogously to \cite{epidemicsSpread} and \cite[Corollary $8.2$]{HdK21}, 
\new{where
$f'(0)$ is again replaced by $f(1)$}. 

\begin{corollary}
Suppose $f$ is concave in the \new{mean field} model given by (\ref{eq:dynamical system})--(\ref{eq:gidef}).
 Let $\tau$ denote the time of extinction of the disease and suppose  
 $\beta \, f(1) \, \lambda(W) / \delta <1$. Then
 \[
 \E[\tau]\leq\frac{1+\log n}{\delta- \beta \, f(1) \, \lambda (W)}.
 \]
\end{corollary}

\section{Computational Experiments}
\label{sec:compute}
\new{In this section we report on results of computational experiments
that allow us to test the sharpness of the results derived in  section~\ref{sec:mf_analysis}, 
and also allow us to 
compare the two mean field models that we have discussed
against each other and against the exact stochastic model.}

\subsection{Simulation algorithm}
First, let us summarize our approach for the mean field approximation (\ref{eq:dynamical system})--(\ref{eq:gidef}). Following \cite{APM20}, we use the discrete Fourier representation of $\Psi(h,l)$ derived in \cite{APM20} to render the computation of (\ref{eq:gidef}) more 
stable.  
We solve the ODE systems 
(\ref{eq:dynamical system})--(\ref{eq:gidef})
and
   (\ref{eq:alternative dynamical system})--(\ref{eq:alternative gidef})
   with Euler's method, using a time step $\Delta t=0.05$. 
   For the exact stochastic model, we use the discretization 
   approach described in \cite{HdK21}.
   The number of nodes is chosen to be $n=400$, and hyperedges of prescribed sizes are generated independently by choosing nodes uniformly at random. In Figure~\ref{fig: atan}, \ref{fig: log}, \ref{fig: min}, \ref{fig: collective}, and \ref{fig: atan and 2log i_infty vs beta}, there are $400$ edges, $200$ hyperedges of size $3$, $100$ hyperedges of size $4$ and $50$ hyperedges of size $50$. The sizes and number of hyperedges differ in Figure~\ref{fig: i_0 discontinuity k>=3}, \ref{fig: i_0 discontinuity k>=4 more hyperedges} and \ref{fig: i_0 discontinuity k>=4 fewer hyperedges}, and are specified in the descriptions of the figures.
\subsection{Experimental Comparisons}

\new{In Figures~\ref{fig: atan},  \ref{fig: log}, \ref{fig: min},
\ref{fig: collective}
and
\ref{fig: atan and 2log i_infty vs beta}}, we compare the time evolution of the two mean field models (\ref{eq:dynamical system})--(\ref{eq:gidef}) and (\ref{eq:alternative dynamical system})--(\ref{eq:alternative gidef}), with the exact model. 
\new{The figures show the proportion of infected individuals:
$\sum_{j=1}^{n} X_j(t)/n$ for the exact model and 
$\sum_{j=1}^{n} p_j(t)/n$ for the mean field models.}
The exact model was run $100$ times independently. \new{The solid green envelopes represent the span of the runs: at each time point we discard the most extreme $10\%$ of the values; that is, $5\%$ of the values above and below the average. 
In these plots, we used the same initial infection probability
$i_0$ for each node, as in 
(\ref{eq:izero}).
The figures give results for different
$i_0$ and
$\beta$ values.}

\new{Figures~\ref{fig: atan} and \ref{fig: log}
use concave infection rates of $\mathrm{arctan}(x)$
and 
$\log(1+x)$, respectively.
Here, both mean field models are seen to 
give good qualitative approximations to the exact models, but it is  noticeable that the model (\ref{eq:alternative dynamical system})--(\ref{eq:alternative gidef})
(red dashed line), which 
applies continuous-valued arguments to $f$, 
overestimates the infection level when $\beta$ and $i_0$ are small and hence the disease vanishes over time.}

\new{Figure~\ref{fig: min} uses another concave infection rate, $f(x) = \min\{3,x\}$. Here, 
both mean field models substantially overestimate the
infection level for small $\beta$ and $i_0$.
It is intuitively reasonable that the two mean field models
behave similarly in this example, since
on hyperedges of size less than or equal to $4$
the infection rate function is linear, and hence commutes with the expectation operation.}

In \new{Figure~\ref{fig: collective}}, we consider a partitioned collective contagion model defined as follows. Letting $f_k$ denote the infection rate function applied to all hyperedges of size $k+1$, we let $f_1:x\mapsto x$, and $f_k:x\mapsto c_{2,k}\1(x\geq c_{1,k})$, $k\in \{2,\dots,4\}$. Here we chose $c_{1,k}=c_{2,k}:=k-1$, for $k\in \{2,\dots,4\}$.
\new{
In this case, the 
 mean field model (\ref{eq:dynamical system})--(\ref{eq:gidef})
(purple dots) fails to predict decay of the disease for small
$\beta$ and $i_0$.}

\new{In Figure~\ref{fig: atan and 2log i_infty vs beta} 
we directly compare the accuracy with which the mean field models predict disease outbreak, as a function of $\beta$, and  
we also test the sharpness of the 
spectral bounds.
Here we use the concave infection rates
$2\log(1+x)$ and $\mathrm{arctan}$. The vanishing conditions predicted by the spectral bounds (\ref{eq:mflocal}) and (\ref{eq:mflocal2}), yielding the green vertical lines in Figure~\ref{fig: atan and 2log i_infty vs beta}, occur at $\beta_1\cong 0.0369$ and $\beta_2\cong 0.0268$ respectively for $f(x)= 2\log(1+x)$, and at $\beta_1\cong 0.0629$ and $\beta_2\cong 0.0494$ respectively, for $f(x)=\arctan(x)$.
With initial infection probability $i_0 = 0.5$
we averaged the infection level at $T = 200$
over $10$ runs.
Blue crosses correspond to the exact model.
We see both mean field models are conservative in the sense that 
they give growth for $\beta$ values where the exact model
produces no infection. The figures also show the spectral 
bounds 
on $\beta$ arising from 
(\ref{eq:mflocal})
and
(\ref{eq:mflocal2})
as vertical lines, and we see that 
they give sharp predictions.
}

\begin{figure}[htp]
    \centering
    \includegraphics[width=\textwidth]{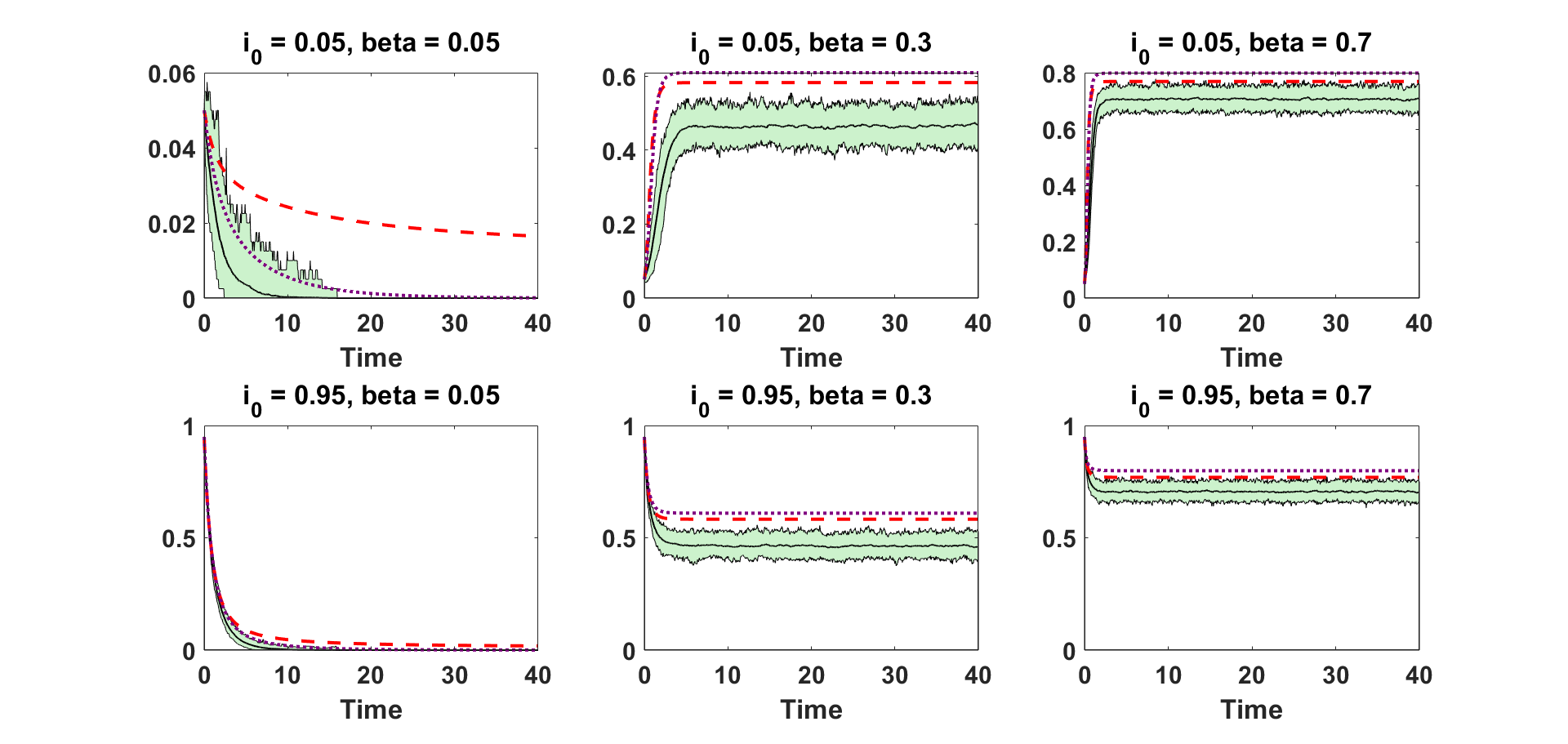}
    \caption{\new{Infection function $f(x) = \mathrm{arctan}(x)$.}
    Purple dots: mean field approximation from (\ref{eq:dynamical system})--(\ref{eq:gidef}).
    Red dashed line: mean field approximation from 
   (\ref{eq:alternative dynamical system})--(\ref{eq:alternative gidef}). 
    Black solid line: mean of the individual-level stochastic model. 
    }
    \label{fig: atan}
\end{figure}


\begin{figure}[htp]
    \centering
    \includegraphics[width=\textwidth]{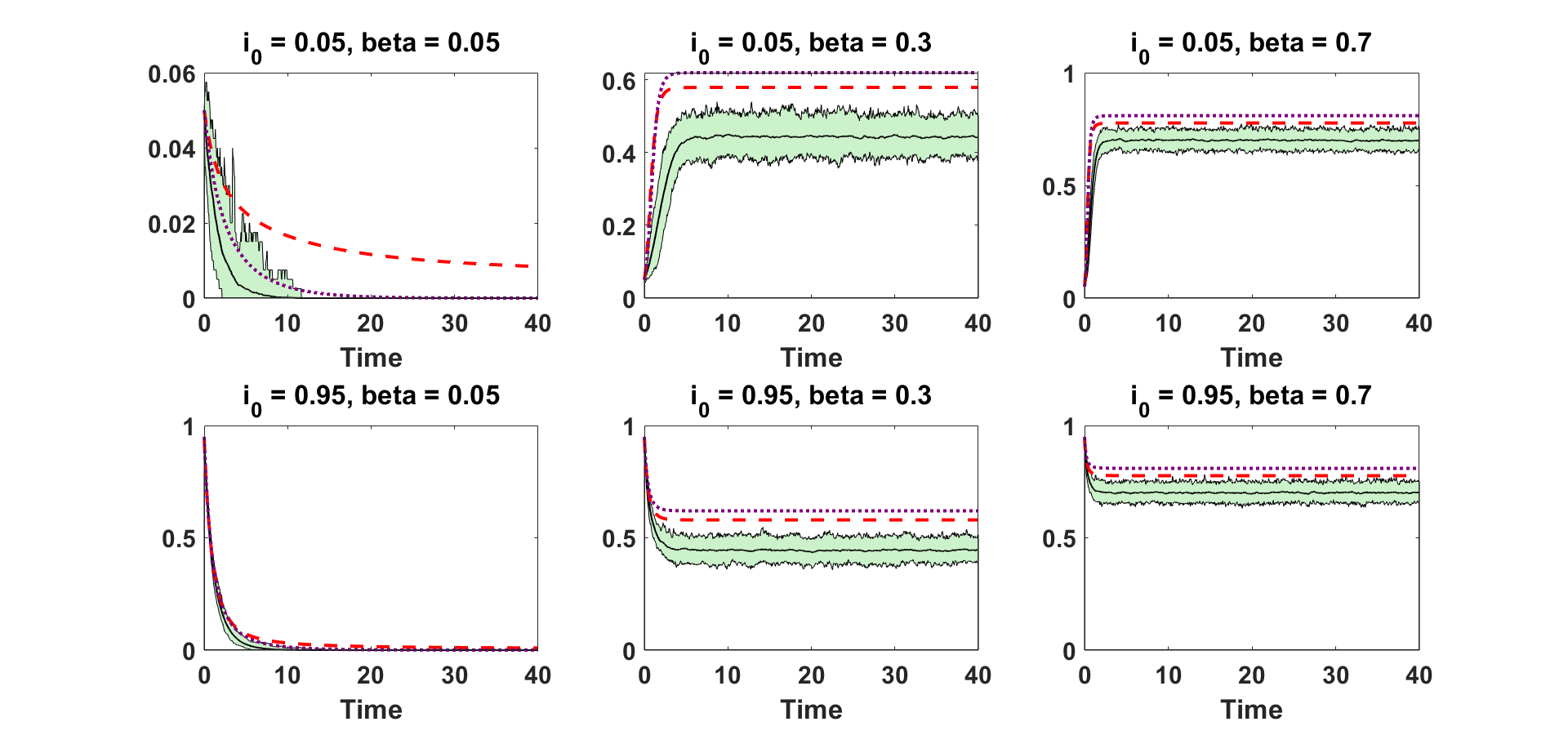}
    \caption{\new{Infection function $f(x) = \log(1+x)$.}
  Purple dots: mean field approximation from (\ref{eq:dynamical system})--(\ref{eq:gidef}).
    Red dashed line: mean field approximation from 
   (\ref{eq:alternative dynamical system})--(\ref{eq:alternative gidef}). 
    Black solid line: mean of the individual-level stochastic model. 
    }
    \label{fig: log}
\end{figure}

\begin{figure}[htp]
    \centering
    \includegraphics[width=\textwidth]{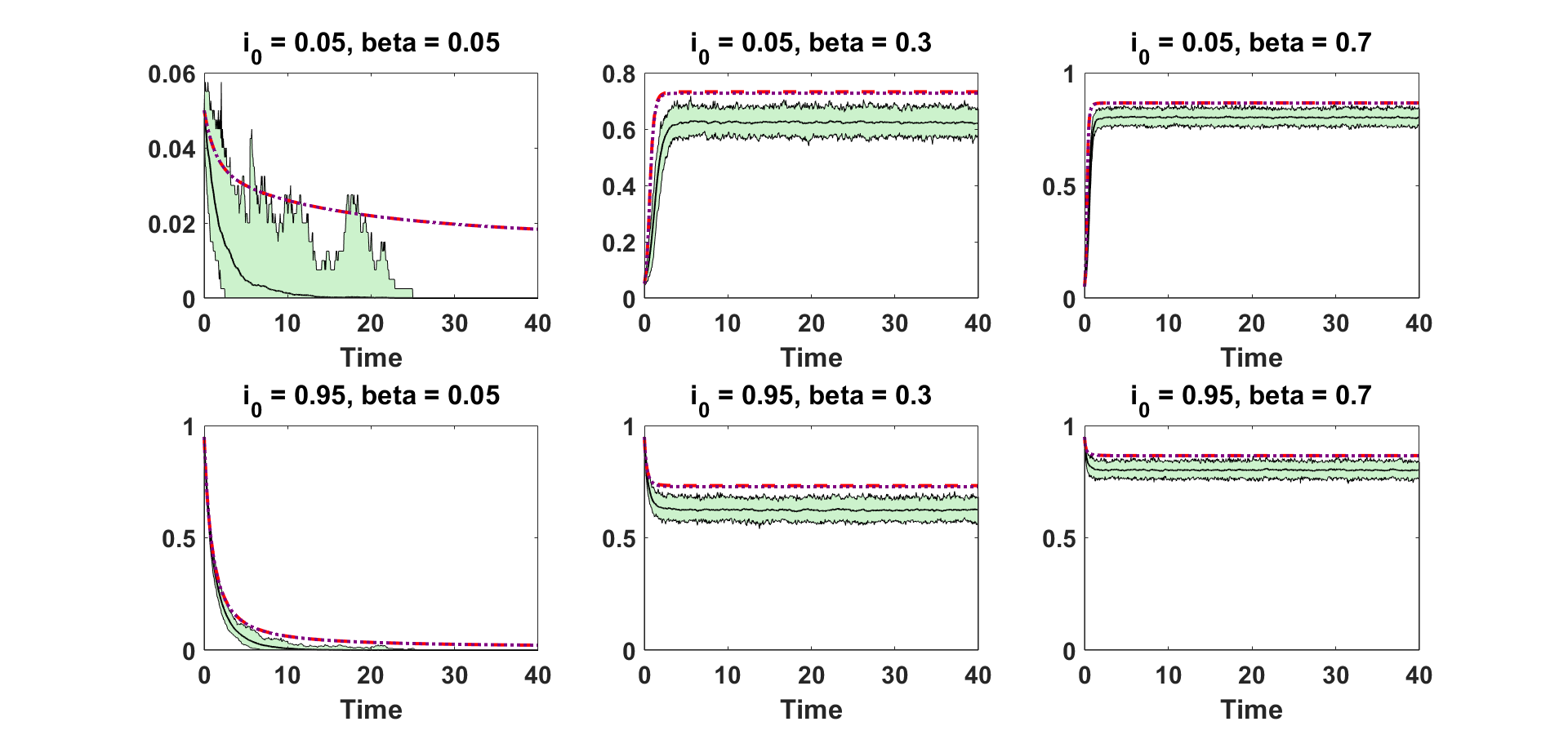}
    \caption{\new{Infection function $f(x) = \min\{3,x\}$.}
    Purple dots: mean field approximation from (\ref{eq:dynamical system})--(\ref{eq:gidef}).
    Red dashed line: mean field approximation from (\ref{eq:alternative dynamical system})--(\ref{eq:alternative gidef}).
    Black solid line: mean of the individual-level stochastic model. 
    }
    \label{fig: min}
\end{figure}

\begin{figure}[htp]
    \centering
    \includegraphics[width=\textwidth]{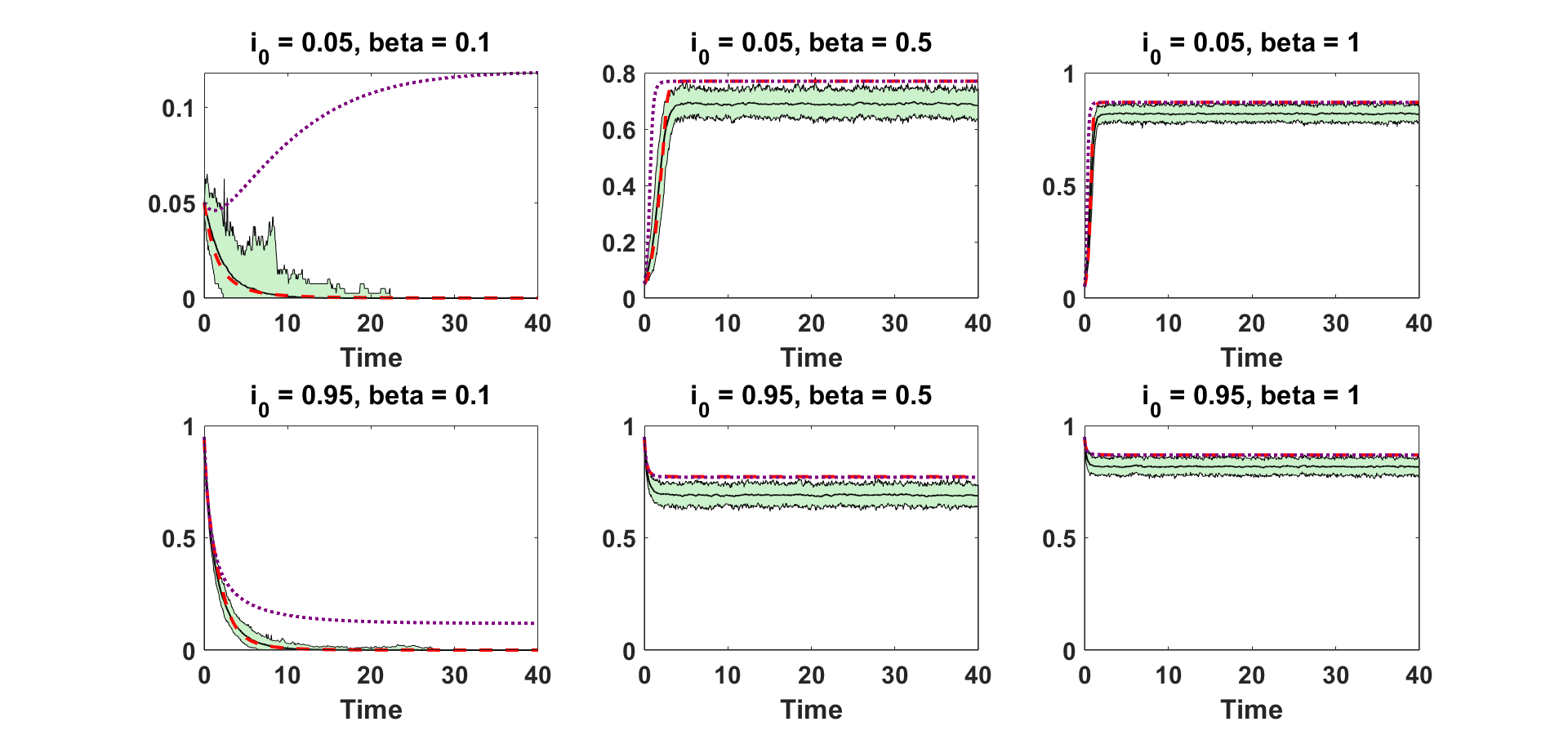}
    \caption{Collective contagion partitioned model. 
   Purple dots: mean field approximation from (\ref{eq:dynamical system})--(\ref{eq:gidef}).
    Red dashed line: mean field approximation from 
   (\ref{eq:alternative dynamical system})--(\ref{eq:alternative gidef}). 
    Black solid line: mean of the individual-level stochastic model. 
    }
    \label{fig: collective}
\end{figure}

\begin{figure}[htp]
    \centering
    \includegraphics[width=0.495\textwidth]{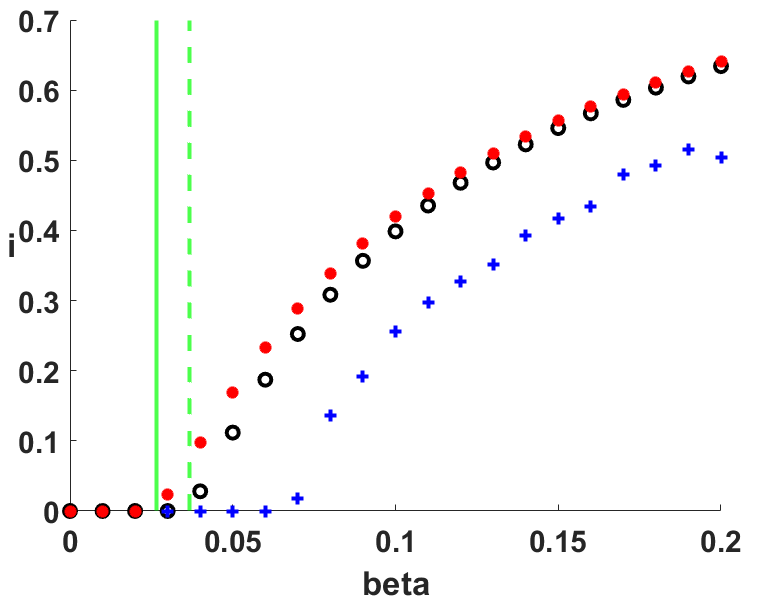}
    \includegraphics[width=0.495\textwidth]{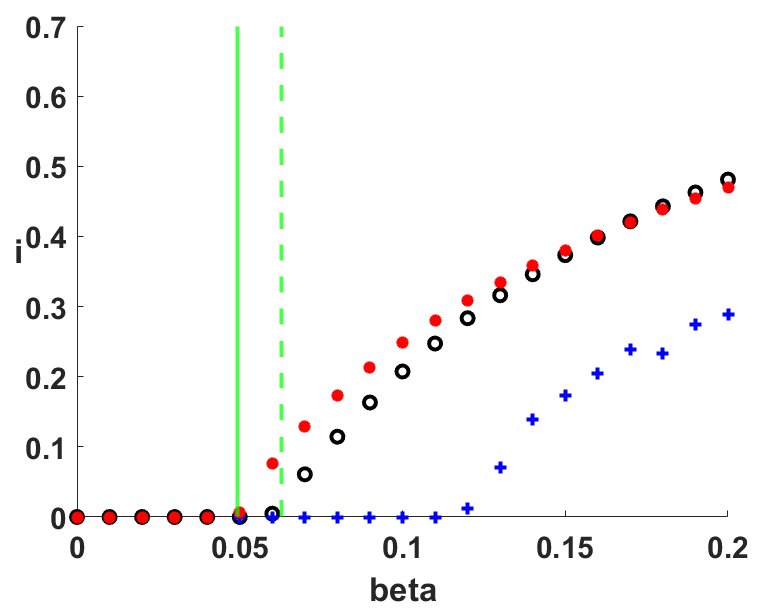}
    \caption{Infection function $2\log(1+x)$ (left) and  $\mathrm{arctan}$ (right).
    Horizontal axis is infection strength, 
      $\beta$. Vertical axis is  
      the proportion of infected individuals at time 
      $ T = 200$ for the two mean 
      field approximations,
   (\ref{eq:alternative dynamical system})--(\ref{eq:alternative gidef}) 
  (red dots) 
  and 
   (\ref{eq:dynamical system})--(\ref{eq:gidef})
    (black circles), and
   for the   
   individual-level stochastic model 
      (blue crosses), averaged over $10$ runs.
      The spectral bounds on $\beta$ 
      from (\ref{eq:mflocal}) and (\ref{eq:mflocal2})
     relating to the two mean field approximations, 
      are shown respectively as a solid green vertical line (below which the red dots must be $0$) and a dashed green  vertical line (below which the black circles must be $0$).
    }
    \label{fig: atan and 2log i_infty vs beta}
\end{figure}

\subsection{Collective contagion model: sensitivity to the initial condition}
\new{An interesting working assumption  
is that only hyperedges of size three or greater are present,
and hence there are no pairwise interactions. 
This circumstance may arise, for example, if we restrict attention to a workplace or school environment.}
 Here we look how this assumption may impact the predictive performance of the two mean field models, in the case of a collective contagion model. 
 We used the same infection rate functions as in 
 Figure~\ref{fig: collective}.
 In Figures~\ref{fig: i_0 discontinuity k>=3}, \ref{fig: i_0 discontinuity k>=4 more hyperedges}, and \ref{fig: i_0 discontinuity k>=4 fewer hyperedges}, we show, for both mean field models and the exact stochastic model, the proportion of infected individuals at time $T=100$ averaged over $5$ runs, as a function of the initial proportion $i_0$ of infected individuals. We observe that the mean field model given by (\ref{eq:dynamical system})--(\ref{eq:gidef}) remains relatively stable, while the behaviour of the mean field model given by (\ref{eq:alternative dynamical system})--(\ref{eq:alternative gidef}) appears to be sensitive to the initial condition $i_0$, its predictive performance degrading if $i_0$ is small (e.g., red dots in Figure~\ref{fig: i_0 discontinuity k>=3}). This sensitivity can be understood intuitively by recalling that the model in (\ref{eq:dynamical system})--(\ref{eq:gidef}) is expressed as a continuous function of $P\in \R^n$, while the model in (\ref{eq:alternative dynamical system})--(\ref{eq:alternative gidef}) is expressed in terms of step functions of the form $\1(\sum_i p_i\mathcal I_{i h}\leq c_1)$; the later are not  continuous functions of $P$ and are more sensitive to small perturbations of the initial condition. Furthermore, for initial value $P(0)=(i_0)_{i=1}^n$ sufficiently small that the threshold conditions of the above step functions are not satisfied, the infection rate expressed by (\ref{eq:alternative dynamical system})--(\ref{eq:alternative gidef}) will remain $0$, while the infection may start to spread according to the other models, thus yielding an underestimate of the propagation of the virus in the population.

We note that if the number of hyperedges is relatively low compared with the number of nodes (as in Figure~\ref{fig: i_0 discontinuity k>=4 fewer hyperedges}), then the exact model will not propagate, in which case the mean field model given by (\ref{eq:alternative dynamical system})--(\ref{eq:alternative gidef}) will give a better prediction. However, we see that both mean field models fail to accurately predict the behaviour of the model for sufficiently large initial condition $i_0$.

\begin{figure}[htp]
    \centering
    \includegraphics[width=0.6\textwidth]{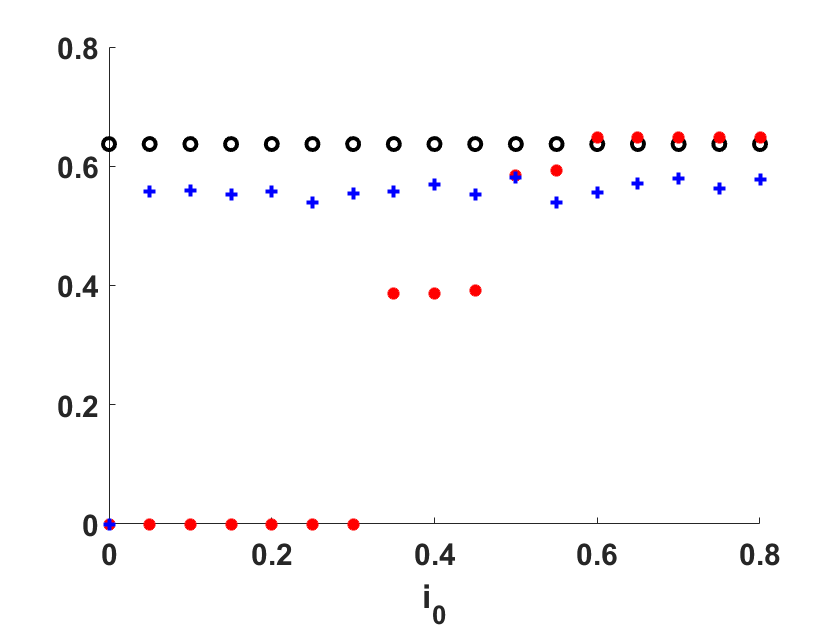}
    \caption{Proportion on infected individuals at time $T=100$ for the two mean field approximation models (red dots for (\ref{eq:alternative dynamical system})--(\ref{eq:alternative gidef}) and black circles for (\ref{eq:dynamical system})--(\ref{eq:gidef})) and the individual-level stochastic model (blue crosses). 
    Using $200$ hyperedges of size $3$, $100$ hyperedges of size $4$, $50$ hyperedges of size $5$ for $400$ nodes.
    }
    \label{fig: i_0 discontinuity k>=3}
\end{figure}

\begin{figure}[htp]
    \centering
    \includegraphics[width=0.6\textwidth]{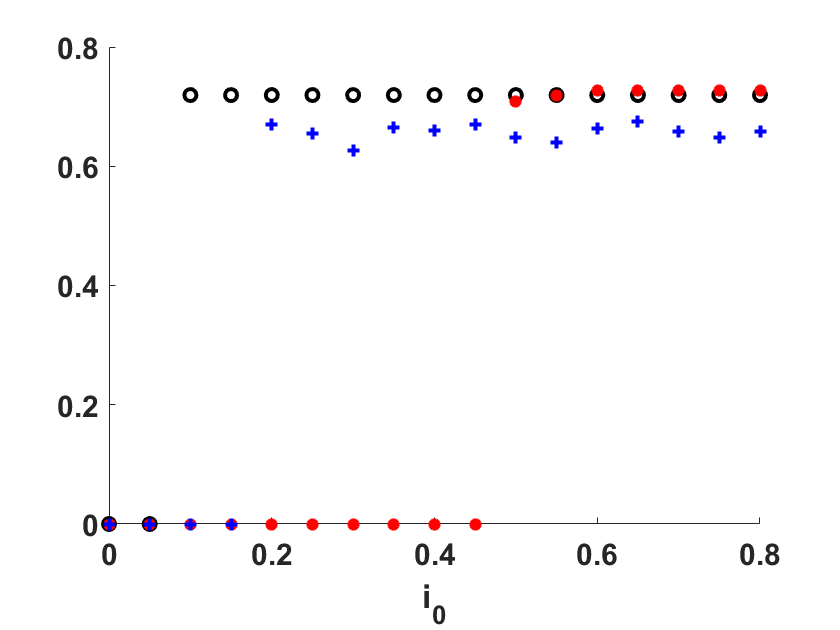}
    \caption{Proportion on infected individuals at time $T=100$ for the two mean field approximation models (red dots for (\ref{eq:alternative dynamical system})--(\ref{eq:alternative gidef}) and black circles for (\ref{eq:dynamical system})--(\ref{eq:gidef})) and the individual-level stochastic model (blue crosses). Using $200$ hyperedges of size $4$, $100$ hyperedges of size $5$ for $400$ nodes.
    }
    \label{fig: i_0 discontinuity k>=4 more hyperedges}
\end{figure}

\begin{figure}[htp]
    \centering
    \includegraphics[width=0.6\textwidth]{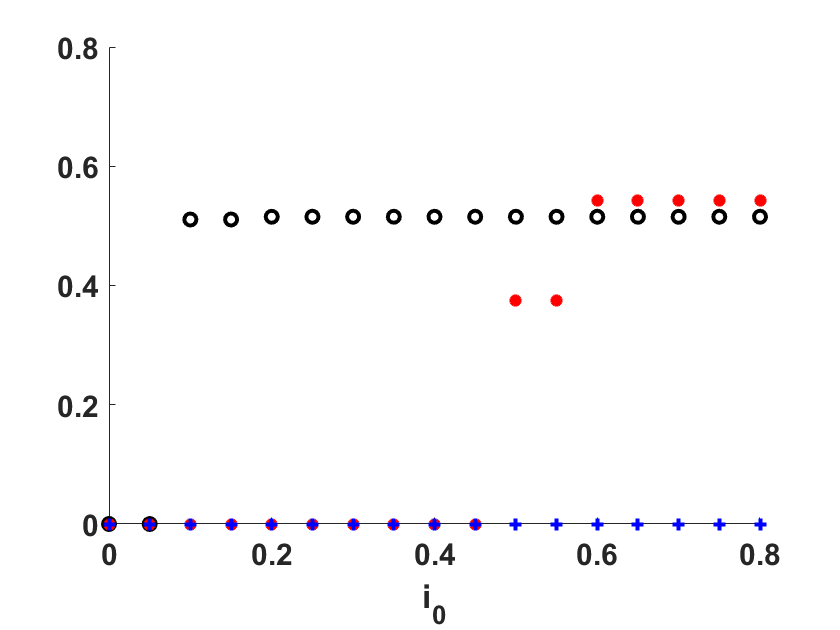}
    \caption{Proportion on infected individuals at time $T=100$ for the two mean field approximation models (red dots for (\ref{eq:alternative dynamical system})--(\ref{eq:alternative gidef}) and black circles for (\ref{eq:dynamical system})--(\ref{eq:gidef})) and the individual-level stochastic model (blue crosses). Using $100$ hyperedges of size $4$, $50$ hyperedges of size $5$ for $400$ nodes.
    }
    \label{fig: i_0 discontinuity k>=4 fewer hyperedges}
\end{figure}


\newpage

\section{Multi-type Model}
\label{sec:multi}
\new{In the above results, we assumed for simplicity that a fixed infection function $f$ applies for all hyperedges. The results, however, readily extend to a \emph{multi-type partition model},
where the infection rate function may depend on the category and size of the hyperedge.
As we discussed in section~\ref{sec:mod}, 
the categories of hyperedge may correspond to locations, such as 
households, schools, offices, shops and public transport vehicles, and 
hyperedge size may have an impact on transmission if individuals
are attempting to mutually distance. 
We will therefore explain how the main results change when we extend
the infection rate model. 
Let us partition the hyperedges of the hypergraph into $S$ disjoint families $\{\mathcal F_s\}_{s=1}^S$, such that to each family $\mathcal F_s$ corresponds an infection function $f_s$. For each $s\in \{1,2,\dots,S\}$ we may further partition the hyperedges in $\mathcal F_s$ into disjoint categories $\mathcal C_2^{(s)},\dots,\mathcal C_{K_s}^{(s)}$, where a hyperedge $h\in \mathcal F_s$ belongs to $\mathcal C_k^{(s)}$ if and only if $|h|=k$.
The infection rate model 
(\ref{eq:Xinf}) may then be extended to 
\begin{equation}
\lambda_i(X(t)) 
= 
\beta 
\,
\sum_{s=1}^{S}
\sum_{k=2}^{K_s}
\sum_{h\in \mathcal C_k^{(s)}}
{\mathcal I}_{i h}^{(s),(k)}
\,
f_s (
\sum_{j=1}^{n} 
{\mathcal I}_{j h}^{(s),(k)}
X_j
),
\label{eq:Xinfmulti}
\end{equation}
where $\mathcal I^{(s),(k)}$ is the incidence matrix inducing the subhypergraph spanned by the hyperedges of $\mathcal C_k^{(s)}\subset \mathcal F_s$, i.e., $\mathcal I^{(s),(k)}_{i h}=1$ if $h\in \mathcal C_k^{(s)}\subset \mathcal F_s$ and $i\in h$, and $\mathcal I^{(s),(k)}_{i h}=0$ otherwise. 
}
 We then have the following generalization of the ODE system in (\ref{eq:dynamical system})--(\ref{eq:gidef})
\begin{equation}
    \label{eq:generalized dynamical system}
    \frac{d P(t)}{dt}=g(P(t)),
\end{equation}
where $g:\R^n\to \R^n$ is defined by
\begin{equation}
    \label{eq:generalized gidef}
    g_i(P(t))=\beta\sum_{s=1}^S\sum_{k=2}^{K_s}\sum_{h\in \mathcal C_k^{(s)}}\mathcal I^{(s),(k)}_{i h}(\sum_{l=1}^kf_s(l)\Psi(h,l))(1-p_i(t))-\delta p_i(t).
\end{equation}
Define also $\mathcal I^{(s)}:=\sum_{k=2}^{K_s}\mathcal I^{(s),(k)}$ to be the incidence matrix inducing the subhypergraph spanned by the hyperedges in $\mathcal F_s$, and let $W^{(s)}:=\mathcal I^{(s)}(\mathcal I^{(s)})^T$, so that $W^{(s)}_{ij}$ records the number of hyperedges in $\mathcal F_s$ containing both $i$ and $j$. We then have the following results for the generalized partition model, which are extensions of Theorems~\ref{thm: local stability}, \ref{thm: global stability concave} and \ref{thm: global stability collective}.

\begin{thm}[General condition for local asymptotic stability]
\label{thm:genralized local asymptotic stability}
If 
\begin{equation}
\frac{ \beta \, 
\lambda\left(\sum_{s=1}^Sf_s(1)W^{(s)}\right)
}{\delta}<1,
\label{eq:as_multi}
\end{equation}
then $0\in \R^n$ is a locally asymptotic stable equilibrium for (\ref{eq:generalized dynamical system})--(\ref{eq:generalized gidef}).
\end{thm}

\begin{thm}[Global asymptotic stability for a collective suppression model]
\label{thm: generalized global stability concave}
Suppose that $f_s$ is concave for all $s\in \{1,2,\dots,S\}$. 
\new{If
(\ref{eq:as_multi}) holds,}
 then $0\in \R^n$ is globally asymptotically stable for (\ref{eq:generalized dynamical system})--(\ref{eq:generalized gidef}).
\end{thm}
 
\begin{thm}[Global asymptotic stability for a collective contagion model]
\label{thm:generalized global stability collective}
Suppose that for each $s\in \{1,2,\dots,S\}$, $f_s(x):=c_{2,s}\1(x\geq c_{1,s})$, where 
\new{$c_{1,s} \ge 2$ and $c_{2,s}>0$}.
If 
\[
\frac{ \beta \, 
\lambda\left(\sum_{s=1}^S\frac{c_{2,s}}{c_{1,s}}W^{(s)}\right)}
{\delta}<1,
\]
then $0\in \R^n$ is globally asymptotically stable for (\ref{eq:generalized dynamical system})--(\ref{eq:generalized gidef}).
\end{thm}

\section{Summary and Conclusions}
\label{sec:conc}

\new{
Hypergraphs offer more flexibility and realism than 
pairwise, graph-based models and they are relevant to many spreading processes where members of a 
    population form groups.  
    In the pairwise setting, with linear infection rates, 
    graph-based models have been widely studied, 
    and spectral stability bounds derived \cite{epidemicsSpread,HS2020,virusSpreadInNetworks,spreadingEigenvalue}.
     Spectral analysis for 
    the hypergraph case was initially developed 
    in \cite{HdK21}, both for an exact individual-level 
    stochastic model and a deterministic mean field approximation.
    In this work we focused on a more sophisticated mean field approximation that was proposed in \cite{APM20} and 
    requires a more detailed analysis. 
    Although this ODE system produces real-valued trajectories, 
    it has the unusual feature of evaluating the 
    nonlinear infection rate function only at integer  
    arguments.
    Intuitively, since the infection function is zero at the origin, 
    this feature is likely to make the approximation more accurate than the version in 
    \cite{HdK21}
    in the case of concave 
    nonlinearity and small infection levels.
    This behaviour was observed in 
     our computational tests 
     (Figures~\ref{fig: atan}--\ref{fig: min}  and 
      Figure~\ref{fig: atan and 2log i_infty vs beta})   
    and is backed by our theoretical 
    analysis---in the concave case, this mean field model produces a locally 
    asymptotically stable disease-free state 
    under the same condition as the exact model
    (see Theorem~\ref{thm: vanishing condition for exact mean field model} with $c_f = f(1)$ and Theorem~\ref{thm: global stability concave}).
    However, for other types of nonlinear infection rate, it is possible for the mean field model in 
     \cite{HdK21} to give a better approximation 
     (Figure~\ref{fig: collective}). 
     Hence the main conclusion from this work is that both 
     mean field models can be analysed 
      rigorously and both can provide useful information.

    It is notable that the spectral conditions for 
    decay of the disease level 
    appearing in our results have the form 
    \[
    \frac{ \beta \, c \, 
\lambda(W) }
{\delta} <1,
\]
    for some constant $c$ that is determined by the 
    type of nonlinear infection rate 
    (with generalized versions in section~\ref{sec:multi}).
    This expression separates out different aspects of the process in a natural manner and offers a means to
    inform mitigation strategies.
    The parameters $\beta$ and $\delta$ quantify the inherent infectiousness
    and recovery rate for the disease, respectively.
    The constant $c$ is affected by the way that the chance of a new infection 
    depends on the number of infected people in a group. This
    could be controlled by changing 
     behavioural patterns; for example, through
      face-covering or social distancing.
      The factor $\lambda(W)$, which 
      summarizes the interaction structure, could be reduced by lockdown measures that restrict movement and therefore 
      limit physical encounters.
      Hence, it would be of interest to calibrate 
      a hypergraph model against real data and investigate the 
      predictive power of these spectral bounds.
    }

\bibliographystyle{siamplain}
\typeout{}
\bibliography{refs}
\end{document}